\documentclass[a4paper]{amsart}

\newtheorem{thm}{Theorem}\numberwithin{thm}{section}
\numberwithin{cor}{section}
\newtheorem{lem}{Lemma}\numberwithin{lem}{section}
\newtheorem{prop}{Proposition}\numberwithin{prop}{section}
\newtheorem*{prop*}{Proposition}
\newtheorem{rem}{Remark}
\numberwithin{rem}{section}

\numberwithin{equation}{section}

\newcommand{\as}{\text{ as }}
\newcommand{\ninf}{n\rightarrow\infty}

\renewcommand{\l}{\left}
\renewcommand{\r}{\right}

\begin{document}

\title[Spectral density zeroes of discrete Schr\"odinger operator]
{Zeroes of the spectral density of
\\
discrete Schr\"odinger operator with
\\
Wigner-von Neumann potential}

\author{Sergey Simonov}
\address{Chebyshev Laboratory, Department of Mathematics and Mechanics, Saint-Petersburg State University, 14th Line, 29b, Saint-Petersburg, 199178 Russia}
\address{Institut f\"ur Analysis und Scientific Computing, Technische Universit\"at Wien, Wiedner Hauptstrasse, 8--10, Wien, A--1040, Austria}
\email{sergey.a.simonov@gmail.com}



\date{\today}

\begin{abstract}
We consider a discrete Schr\"odinger operator $\mathcal J$ whose potential is the sum of a Wigner-von Neumann term $\frac{c\sin(2\omega n+\delta)}n$ and a summable term. The essential spectrum of the operator $\mathcal J$ equals to the interval $[-2,2]$. Inside this interval, there are two critical points $\pm2\cos\omega$ where eigenvalues may be situated. We prove that, generically, the spectral density of $\mathcal J$ has zeroes of the power $\frac{|c|}{2|\sin\omega|}$ at these points.
\end{abstract}

\maketitle

\section{Introduction}
In the present paper we consider a discrete Schr\"odinger operator, i.e., a Jacobi matrix
\begin{equation}\label{J}
    \mathcal J=
    \left(%
    \begin{array}{cccc}
    b_1 & 1 & 0 & \cdots \\
    1 & b_2& 1 & \cdots \\
    0 & 1 & b_3 & \cdots \\
    \vdots & \vdots & \vdots & \ddots \\
    \end{array}%
    \right),
\end{equation}
whose diagonal entries (potential) are of the form
\begin{equation}\label{b}
    b_n:=\frac{c\sin(2\omega n+\delta)}n+q_n,
\end{equation}
where $c,\omega,\delta$ are real constants and $\{q_n\}_{n=1}^{\infty}$ is a real-valued sequence such that
\begin{equation}\label{conditions}
    c\neq0,\omega\notin\frac{\pi\mathbb Z}2\text{ and }\{q_n\}_{n=1}^{\infty}\in l^1.
\end{equation}

The operator $\mathcal J$ is a compact perturbation of the free discrete Schr\"odinger operator and therefore, by Weyl's theorem, its essential spectrum equals the interval $[-2,2]$, cf. \cite{Birman-Solomyak-1987}. Moreover, since $\{b_n\}_{n=1}^{\infty}\in l^2$, the interval $(-2,2)$ is covered with absolutely continuous spectrum, cf. \cite{Deift-Killip-1999}. The presence of  the (discrete) Wigner-von Neumann potential $\frac{c\sin(2\omega n+\delta)}n$ with frequency $\omega$ produces two critical (resonance) points, namely, the points $\pm2\cos\omega$ where $\mathcal J$ may have half-bound states or eigenvalues. Here, by a half-bound state we understand a point where a subordinate solution of the eigenfunction equation exists, which does not belong to $l^2$. The spectrum on the rest of the interval $(-2,2)$ is purely absolutely continuous, cf. \cite{Janas-Simonov-2010}.

In the present paper we study the behaviour of the spectral density of the operator $\mathcal J$ near the critical points $\pm2\cos\omega$. Our main result is Theorem \ref{thm main result}, where we show that the spectral density has a zero of the order $\frac{|c|}{2|\sin\omega|}$ at each of these points, unless an eigenvalue or a half-bound state is located there.

Vanishing of the spectral density divides the absolutely continuous spectrum into separate parts and is called pseudogap. The physical meaning of this phenomenon is that the interval of the spectrum near such point contains very few energy levels.

The proof of Theorem \ref{thm main result} is based on two ingredients. The first is a Weyl-Titchmarsh type formula taken from \cite{Janas-Simonov-2010}, which relates the value of the spectral density to the coefficient in the orthogonal polynomials asymptotics (Proposition \ref{prop Janas-Simonov}). The second is the limit behaviour of the solutions of a certain model discrete linear system (Proposition \ref{prop Naboko-Simonov}), which has been studied in \cite{Naboko-Simonov-2011}.

Operators with Wigner-von Neumann potentials \cite{Wigner-von-Neumann-1929} attracted attention of many authors \cite{Albeverio-1972, Matveev-1973, Burd-Nesterov-2010,Lukic-2011}. In \cite{Hinton-Klaus-Shaw-1991} the Weyl function behaviour near the critical points was studied for the differential Schr\"odinger operator on the half-line with an infinite sum of Wigner-von Neumann terms in the potential. The spectral density is proportional to the boundary value of imaginary part of the Weyl function. Hence the object under consideration in \cite{Hinton-Klaus-Shaw-1991} is the same as in the present paper. The questions addressed in \cite{Hinton-Klaus-Shaw-1991} were also studied by a different method in \cite{Behncke-1991-I,Behncke-1991-II,Behncke-1994}. In \cite{Klaus-1991} all the possible cases for such potentials were considered (bound state or half-bound state).

In \cite{Kurasov-Simonov-2011,Naboko-Simonov-2011} (following \cite{Kurasov-Naboko-2007}) we considered the Schr\"odinger operator with a potential which is the sum of Wigner-von Neumann part, a summable part, and a periodic background part. There we have proposed a new approach based on the study of the model discrete linear system \eqref{model system}. We consider that the main advantage of this approach is that the result can be formulated as a theorem concerning the model system. This leads to a greater universality of the method. In the present paper we show that it is applicable to the Jacobi matrix case and allows to use the same theorem as in \cite{Naboko-Simonov-2011}. We plan to consider other applications of this method, for instance, the Schr\"odinger operator on the half-line with point interactions supported by a lattice (Kronig-Penney model) with the sequence of interacting centers or strengths of interaction perturbed by a sequence of Wigner-von Neumann potential form \cite{Lotoreichik-Simonov-2012}.

The paper is organised as follows. In Section \ref{section preliminaries} we define the operator $\mathcal J$ and recall the Weyl-Titchmarsh type formula for its spectral density. In Section \ref{section reduction} we transform the eigenfunction equation to a form of the model discrete linear system \eqref{system for v-hat+}. In Section \ref{section Naboko-Simonov} we recall the results about this system which were obtained in \cite{Naboko-Simonov-2011}. In Section \ref{section final result} we prove our main result, Theorem \ref{thm main result}.

\section{Preliminaries}\label{section preliminaries}
The operator $\mathcal J$ acts in the Hilbert space $l^2$ of square summable complex-valued sequences by the rule
\begin{equation}\label{action of J}
    \begin{array}{l}
    (\mathcal{J}u)_1=b_1u_1+u_2, \\
    (\mathcal{J}u)_n=u_{n-1}+b_nu_n+u_{n+1},\ n\ge2. \\
    \end{array}
\end{equation}
on the domain
\begin{equation*}
    \mathcal D(\mathcal J)=\{u\in l^2:\text{ the result of \eqref{action of J} is in }l^2\}
\end{equation*}
(maximal domain) and is self-adjoint \cite{Akhiezer-1965}. It has a matrix representation of the form \eqref{J} in the canonical base of $l^2$.
Eigenfunction equation for $\mathcal J$ is the following three-term recurrence relation:
\begin{equation}\label{spectral equation prelim}
    u_{n-1}+b_nu_n+u_{n+1}=\lambda u_n,\ n\ge2,
\end{equation}
and its solutions are called generalized eigenvectors. One of these solutions is formed of polynomials $P_n(\lambda)$ which additionally satisfy the "first line" equation: $b_1u_1+u_2=\lambda u_1$, and are defined by conditions $P_1(\lambda)=1,P_2(\lambda)=\lambda-b_1$. There exists a measure $\rho$ such that polynomials $P_n(\lambda)$, $n\in\mathbb N$, form an orthogonal base in the space $L_2(\mathbb R,\rho)$. Moreover, the operator $\mathcal J$ is unitarily equivalent to the operator of multiplication by an independent variable in this space, and so the measure $\rho$ is called the spectral measure of $\mathcal J$. The derivative of the spectral measure $\rho'$ is called the spectral density and is the main object of our interest.

The spectral density of the discrete Schr\"odinger operator with summable potential can be expressed in terms of the asymptotics as $\ninf$ of its orthogonal polynomials by the Weyl-Titchmarsh (Kodaira) type formula. The classical Weyl-Titchmarsh formula deals with the differential Schr\"odinger operator on the half-line with summable potential \cite[Chapter 5]{Titchmarsh-1946-1}, \cite{Kodaira-1949}. For the operator $\mathcal J$ considered in this paper (and actually for a larger class of discrete Scr\"odinger operators with non-summable potentials) analogue of it was obtained in \cite{Janas-Simonov-2010} and is given by the following statement. Define the new variable $z$ as follows:
\begin{equation*}
    \lambda=z+\frac1z\text{ and }z=\frac{\lambda+i\sqrt{4-\lambda^2}}2.
\end{equation*}
The interval $[-2,2]$ of the variable $\lambda$ corresponds to the upper half of the unit circle of the variable $z$.

    \begin{prop}[Janas-Simonov]\label{prop Janas-Simonov}
    Let $\mathcal J$ be the discrete Schr\"odinger operator with the potential $\{b_n\}_{n=1}^{\infty}$ given by \eqref{b} and let conditions \eqref{conditions} hold. Then there exists a continuous function $F:\mathbb T\backslash\{1,-1,e^{\pm i\omega},-e^{\pm i\omega}\}\rightarrow\mathbb C$ such that orthogonal polynomials associated to $\mathcal J$ have the following asymptotics for  $\lambda\in(-2,2)\backslash\{\pm2\cos\omega\}$:
    \begin{equation*}
    P_n(\lambda)=\frac{zF(z)}{1-z^2}\cdot\frac1{z^n}+\frac{z\overline{F(z)}}{z^2-1}\cdot z^n+o(1)\text{ as }n\rightarrow\infty.
    \end{equation*}
    Function $F$ does not vanish on $\mathbb T\backslash\{1,-1,e^{\pm i\omega},-e^{\pm i\omega}\}$. Spectrum of $\mathcal J$ is purely absolutely continuous on $(-2;2)\backslash\{\pm2\cos\omega\}$. The spectral density of $\mathcal J$ equals:
    \begin{equation}\label{Weyl-Titchmarsh type formula}
        \rho'(\lambda)=\frac{\sqrt{4-\lambda^2}}{2\pi|F(z)|^2},\ \lambda\in(-2;2).
    \end{equation}
    \end{prop}

The Weyl-Titchmarsh type formula \eqref{Weyl-Titchmarsh type formula} will be the main tool in our analysis of the behaviour of the spectral density.

\section{Reduction of the eigenfunction equation to the model problem}\label{section reduction}


In this section we transform the eigenfunction equation for $\mathcal J$ rewriting it as a discrete linear system in $\mathbb C^2$ and reducing it to the model system of a simple form, which was studied in \cite{Naboko-Simonov-2011}. As the result we will be able to control the spectral density of $\mathcal J$ by a reformulation of Proposition \ref{prop Janas-Simonov} above in terms of a certain solution of the model system. As a byproduct we establish the asymptotic behavior of generalized eigenvectors at the critical points.

Consider for $\lambda\in(-2;2)$ the eigenfunction equation for $\mathcal J$
\begin{equation}\label{spectral equation}
    u_{n-1}+b_nu_n+u_{n+1}=\lambda u_n,\ n\ge2.
\end{equation}
Write it in the vector form,
\begin{equation}\label{spectral equation, vector form}
    \left(%
        \begin{array}{c}
        u_n \\
        u_{n+1} \\
        \end{array}%
    \right)
    =\left(%
        \begin{array}{cc}
        0 & 1 \\
        -1 & \lambda-b_n \\
        \end{array}%
    \right)
    \left(%
        \begin{array}{c}
        u_{n-1} \\
        u_n \\
        \end{array}%
    \right),
    \ n\ge 2.
\end{equation}
Consider a new parameter $\phi$ such that $\lambda=2\cos\phi$.
Variation of parameters in the form
\begin{equation*}
    \left(%
    \begin{array}{c}
    u_n \\
    u_{n+1} \\
    \end{array}%
    \right)
    =
    \left(%
    \begin{array}{cc}
    e^{-i\phi n} & e^{i\phi n} \\
    e^{-i\phi(n+1)} & e^{i\phi(n+1)} \\
    \end{array}%
    \right)
    v_n
\end{equation*}
in the system \eqref{spectral equation, vector form} leads to an equivalent system
\begin{equation}\label{system for v}
    v_{n+1}=M_n(\phi)v_n,\ n\ge1,
\end{equation}
with the coefficient matrix
\begin{multline}\label{M}
    M_n(\phi):=I+\frac{b_{n+1}}{2i\sin\phi}
    \left(%
    \begin{array}{cc}
    1 & e^{2i\phi(n+1)} \\
    -e^{-2i\phi(n+1)} & -1 \\
    \end{array}%
    \right)
    \\
    =I+V^{(1)}_n(\phi)+R^{(1)}_n(\phi),
\end{multline}
where
\begin{multline}\label{V-1}
    V^{(1)}_n(\phi):=
    \frac{c\sin(2\omega(n+1)+\delta)}{2i(n+1)\sin\phi}
    \left(%
    \begin{array}{cc}
    1 & 0 \\
    0 & -1 \\
    \end{array}%
    \right)
    \\
    +
    \frac{c}{4(n+1)\sin\phi}
    \left(%
    \begin{array}{cc}
    0 & e^{i(2(\phi-\omega)(n+1)-\delta)} \\
    e^{-i(2(\phi-\omega)(n+1)-\delta)} & 0 \\
    \end{array}%
    \right)
    \\
    -
    \frac{c}{4(n+1)\sin\phi}
    \left(%
    \begin{array}{cc}
    0 & e^{i(2(\phi+\omega)(n+1)+\delta)} \\
    e^{-i(2(\phi+\omega)(n+1)+\delta)} & 0 \\
    \end{array}%
    \right)
\end{multline}
and
\begin{equation}\label{R-1}
    R^{(1)}_n(\phi):=\frac{q_{n+1}}{2i\sin\phi}
    \left(%
    \begin{array}{cc}
    1 & e^{i\phi(n+1)} \\
    -e^{-i\phi(n+1)} & -1 \\
    \end{array}%
    \right).
\end{equation}
The following theorem gives asymptotics of generalized eigenvectors of $\mathcal J$ for different values of the spectral parameter belonging to the interval $(-2,2)$.

    \begin{thm}\label{thm asymptotics of GEV}
    Let $\mathcal J$ be the discrete Schr\"odinger operator with the potential $\{b_n\}_{n=1}^{\infty}$ given by \eqref{b}, $\{q_n\}_{n=1}^{\infty}$ be real-valued sequence such that $\{q_n\}_{n=1}^{\infty}\in l^1$ and let the conditions \eqref{conditions} hold. Then for every $\lambda\in(-2,2)$ there exists a base $u^+(\lambda)$ and $u^-(\lambda)$ of generalized eigenvectors of $\mathcal J$ with the following asymptotics as $\ninf$.
    \\
    1. For $\lambda=2\cos\omega$
    \begin{equation}\label{asymptotics at omega}
    \begin{array}{l}
        u^+_n(\lambda)=n^{\frac{c}{4\sin\omega}}(\cos(\omega n+\delta/2)+o(1)),
        \\
        u^-_n(\lambda)=n^{-\frac{c}{4\sin\omega}}(\sin(\omega n+\delta/2)+o(1)).
    \end{array}
    \end{equation}
    2. For $\lambda=-2\cos\omega$
    \begin{equation}\label{asymptotics at -omega}
    \begin{array}{l}
        u^+_n(\lambda)=(-1)^nn^{\frac{c}{4\sin\omega}}(\sin(\omega n+\delta/2)+o(1)),
        \\
        u^-_n(\lambda)=(-1)^nn^{-\frac{c}{4\sin\omega}}(\cos(\omega n+\delta/2)+o(1)).
    \end{array}
    \end{equation}
    3. For $\lambda=2\cos\phi\in(-2;2)\backslash\{\pm2\cos\omega\}$
    \begin{equation*}
    \begin{array}{l}
        u^+_n(\lambda)=\exp(i\phi n)+o(1),
        \\
        u^-_n(\lambda)=\exp(-i\phi n)+o(1).
    \end{array}
    \end{equation*}
    \end{thm}

   \begin{rem}
   A similar result is obtained in, e.g.,  \cite{Nesterov-2010} by the method of averaging (and \cite{Janas-Simonov-2010} for the third case). Note that in $(-2,2)\backslash\{\pm2\cos\omega\}$ eigenfunction equation is elliptic (i.e., all its solutions have the same order of magnitude), while at the critical points $\pm2\cos\omega$ it is hyperbolic (the orders of two solutions are different).
    \end{rem}

\begin{proof}
For $\lambda\in(-2,2)$ the transfer matrix \eqref{M} has the form
\begin{equation*}
    M_n(\phi)=I+\frac{c}{4n\sin\omega}X_j+V_{j,n}(\phi),
\end{equation*}
where $\{V_{j,n}(\phi)\}_{n=1}^{\infty}$ is some conditionally summable matrix-valued sequence which belongs to $l^2$ and $X_j$ is a constant matrix, which is different in the three cases under consideration:
\\
1. $\lambda=2\cos\omega$ or $\phi=\omega$: $X_1=
\left(
  \begin{array}{cc}
    0 & e^{-i\delta} \\
    e^{i\delta} & 0 \\
  \end{array}
\right)$.
\\
2. $\lambda=-2\cos\omega$ or $\phi=\omega+\pi$: $X_2=
\left(
  \begin{array}{cc}
    0 & -e^{-i\delta} \\
    -e^{i\delta} & 0 \\
  \end{array}
\right)$.
\\
3. $\lambda\in(-2,2)\backslash\{\pm2\cos\omega\}$ or $\phi\in(0,\pi)\backslash\{\omega\mod\pi,\pi-(\omega\mod\pi)\}$: $X_3=0$.
\\
By \cite[Theorem 3.2]{Benzaid-Lutz-1987} (a version of the discrete Levinson theorem) we can neglect the term $V_n(\phi)$ (i.e., solutions of the system \eqref{system for v} have the same asymptotics as solutions of the analogous system without this term). Hence in all three cases $j=1,2,3$ the system $v_{n+1}=M_n(\phi)v_n$ has a base of solutions $v^{(1)}_n$ and $v^{(2)}_n$ with the following asymptotics as $\ninf$:
\begin{equation*}
    v^{(1)}_n=n^{\frac{c\mu_1(X_j)}{4\sin\omega}}\left(\overrightarrow x_j^{(1)}+o(1)\right),
\end{equation*}
and
\begin{equation*}
    v^{(2)}_n=n^{\frac{c\mu_2(X_j)}{4\sin\omega}}\left(\overrightarrow x_j^{(2)}+o(1)\right),
\end{equation*}
where $\mu_1(X_j),\mu_2(X_j)$ are the eigenvalues of matrices $X_j$ and $\overrightarrow x_j^{(1)},\overrightarrow x_j^{(2)}$ are the corresponding eigenvectors. Returning to the solution $u_n$ by the equality $u_n=e^{-i\phi n}(v_n)_1+e^{i\phi n}(v_n)_2$ we complete the proof (by $(v_n)_1$ and $(v_n)_2$ we denote two components of the vector $v_n\in\mathbb C^2$).
\end{proof}

From now on we will use the expression $c\sin(2\omega n+\delta)$ in the rewritten form $$c\sin(2\omega n+\delta)=|c|\sin(2\omega_1n+\delta_1)$$ with \begin{equation}\label{delta-1}
\delta_1:=\delta+\frac{\pi}2(\text{sign}\,c-1)
\end{equation}
and
\begin{equation}\label{omega-1}
\omega_1:=\omega-\pi\left\lfloor\frac{\omega}{\pi}\right\rfloor\in(0,\pi),
\end{equation}
where $\lfloor\cdot\rfloor$ denotes the standard floor function ($\lfloor x\rfloor$ is the greatest integer which is less than $x$). Now we can fix the range of the variable $\phi\in(0,\pi)$ corresponding to $\lambda\in(-2,2)$, so that $z=e^{i\phi}$.

The term $V^{(1)}_n(\phi)$ for $\phi\neq\omega_1,\pi-\omega_1$ is conditionally summable and belongs to $l^2$. As Theorem \ref{thm asymptotics of GEV} shows, this term does not affect the type of asymptotics of solutions (this leads to the preservation of the absolutely continuous spectrum and was considered in detail in \cite{Janas-Simonov-2010}). The values $\phi=\omega_1,\pi-\omega_1$ correspond to $\lambda=\pm2\cos\omega$, i.e., to the resonance points. At these points the term
$V^{(1)}_n(\phi)$ is not summable even conditionally and the type of solutions asymptotics is different.

As Proposition \ref{prop Janas-Simonov} and the forumla \eqref{system for v} suggest, the spectral density is related to the asymptotic behavior of the solution to the system $v_{n+1}=M_n(\phi)v_n$, which corresponds to orthogonal polynomials. We need to understand the dependence of asymptotics of this solution on the parameter $\lambda$ (or, equivalently, on the parameter $\phi$) near two critical points. The analysis is based upon the idea that if (for example) $\phi$ is close to $\omega_1$, then
\begin{equation*}
    M_n(\phi)=I+\frac{|c|}{4n\sin\omega_1}
    \left(
      \begin{array}{c}
        0 \qquad e^{i(2(\phi-\omega_1)n-\delta_1)} \\
        e^{-i(2(\phi-\omega_1)n-\delta_1)} \qquad 0 \\
      \end{array}
    \right)
    +\left(\begin{array}{c}
	\text{some}
	\\
	\text{inessential part}
	\end{array}
	\right).
\end{equation*}
As we have seen before, the terms in matrix entries of $M_n(\phi)$ of the form $\frac{e^{i\alpha n}}n$ are "dangerous" (make effect on asymptotics) only if $\alpha\in2\pi\mathbb Z$. In the case $\alpha=0$ the term of the type $\frac Xn$, where $X$ is some constant matrix, produces a resonance (change of solutions asymptotics). Now we want to eliminate all non-resonating exponential terms from the system by a certain transformation, i.e., by substitution $v_n\mapsto w_n:=T_n(\phi)v_n$, where $\{T_n(\phi)\}_{n=1}^{\infty}$ is a sequence of invertible matrices. Such a substitution leads to the discrete linear system $w_{n+1}=T_{n+1}^{-1}M_n(\phi)T_nw_n$. Transformations that we find are local, i.e., exist and can be applied only in some neighbourhoods of the critical points. It is important to control the properties of the summable remainder to ensure that it is still uniformly summable after the transformation. Let us introduce the following notation. Let $S$ be some subset of the complex plane and $R_n(\lambda)$,   $n\in\mathbb N,\lambda\in S$ be a sequence of $2\times2$ matrices depending on a parameter. We write $\{R_n(\lambda)\}_{n=1}^{\infty}\in l^1(S)$, if there exists a sequence of positive numbers $\{r_n\}_{n=1}^{\infty}\in l^1$ such that for every $\lambda\in S$ and $n\in\mathbb N$ one has $\|R_n(\lambda)\|<r_n$.

Let $U_+$ and $U_-$ be open intervals such that
\begin{equation*}
    \begin{array}{rl}
    \omega_1\in &U_+\subset(0,\pi)\backslash\{\pi-\omega_1\},
    \\
    \pi-\omega_1\in &U_-\subset(0,\pi)\backslash\{\omega_1\}.
    \end{array}
\end{equation*}
Define Harris-Lutz type transformations \cite{Harris-Lutz-1975,Benzaid-Lutz-1987} as follows:
\begin{equation}\label{T-pm}
    T^{\pm}_n(\phi):=-\sum_{k=n}^{\infty}
    \Bigl[
    V^{(1)}_k(\phi)\mp\frac{|c|}{4k\sin\omega_1}
    \left(
      \begin{array}{cc}
        0 & e^{i(2(\phi\mp\omega_1)k\mp\delta_1)} \\
        e^{-i(2(\phi\mp\omega_1)k\mp\delta_1)} & 0 \\
      \end{array}
    \right)
    \Bigr].
\end{equation}

We will later use the following (trivial) result.

    \begin{lem}\label{lem estimate of sum for Harris-Lutz}
    For every real $\xi\in\mathbb R\backslash2\pi\mathbb Z$ and $n\in\mathbb N$ one has
    $\l|\sum\limits_{k=n}^{\infty}\frac{e^{ik\xi}}k\r|\le\frac1{n\l|\sin\frac{\xi}2\r|}$.
    \end{lem}

\begin{proof}
Straightforwardly,
\begin{multline*}
    \left|(e^{i\xi}-1)\sum_{k=n}^{\infty}\frac{e^{ik\xi}}k\right|
    =\left|\sum_{k=n}^{\infty}e^{i(k+1)\xi}\l(\frac1k-\frac1{k+1}\r)-\frac{e^{in\xi}}n\right|
    \\
    \le\sum_{k=n}^{\infty}\l(\frac1k-\frac1{k+1}\r)+\frac1n=\frac2n.
\end{multline*}
Since $|e^{i\xi}-1|=2\l|\sin\frac{\xi}2\r|$, the proof is complete.
\end{proof}

Now we are able to state the properties of the transformations $T^{\pm}$.

    \begin{lem}\label{lem properties of Harris-Lutz transformation}
    Sums in \eqref{T-pm}, which define $T^{\pm}(\phi)$, converge in $U_{\pm}$ and estimates
    \begin{equation}\label{estimates of T-pm}
        T^{\pm}_n(\phi)=O\l(\frac1n\r)\as\ninf
    \end{equation}
    hold uniformly in $U_{\pm}$, respectively. Moreover,
    \begin{multline}\label{Harris-Lutz transformation T-pm}
        \exp(-T^{\pm}_{n+1}(\phi))M_n(\phi)\exp(T^{\pm}_n(\phi))
        \\
        =I\pm\frac{|c|}{4n\sin\omega_1}
        \left(%
        \begin{array}{cc}
        0 & e^{i(2(\phi\mp\omega_1)n\mp\delta_1)} \\
        e^{-i(2(\phi\mp\omega_1)n\mp\delta_1)} & 0 \\
        \end{array}%
        \right)
        +R_n^{\pm}(\phi),
    \end{multline}
    where $\{R_n^{\pm}(\phi)\}_{n=1}^{\infty}\in l^1(U_{\pm})$ and for every natural $n$ functions $R_n^{\pm}(\cdot)$ are 				continuous in $U_{\pm}$, respectrively.
    \end{lem}

\begin{proof}
Let us prove the statement for $T^+$ (one can obtain the proof of the second statement by changing the notation). Write
\begin{equation*}
    T^+_n(\phi)=\sum_{k=n}^{\infty}t^+_k(\phi),
\end{equation*}
where
\begin{multline*}
    t^+_k(\phi):=
    -\frac{|c|\sin(2\omega_1(k+1)+\delta_1)}{2i(k+1)\sin\phi}
    \left(%
    \begin{array}{cc}
    1 & 0 \\
    0 & -1 \\
    \end{array}%
    \right)
    \\
    -
    \frac{|c|}{4(k+1)\sin\phi}
    \left(%
    \begin{array}{cc}
    0 & e^{i(2(\phi-\omega_1)(k+1)-\delta_1)} \\
    e^{-i(2(\phi-\omega_1)(k+1)-\delta_1)} & 0 \\
    \end{array}%
    \right)
    \\
    +
    \frac{|c|}{4(k+1)\sin\phi}
    \left(%
    \begin{array}{cc}
    0 & e^{i(2(\phi+\omega_1)(k+1)+\delta_1)} \\
    e^{-i(2(\phi+\omega_1)(k+1)+\delta_1)} & 0 \\
    \end{array}%
    \right)
    \\
    +
    \frac{|c|}{4k\sin\omega_1}
    \left(%
    \begin{array}{cc}
    0 & e^{i(2(\phi-\omega_1)k-\delta_1)} \\
    e^{-i(2(\phi-\omega_1)k-\delta_1)} & 0 \\
    \end{array}%
    \right).
\end{multline*}
Since the values $\xi=\pm2\omega_1,\pm2(\phi+\omega_1)$ do not belong to $\mathbb R\backslash2\pi\mathbb Z$ for $\phi\in U_+$, the sum over $k$ of first and third terms can be uniformly estimated using Lemma \ref{lem estimate of sum for Harris-Lutz}. The difference between the second term and the same expression with $k$ instead of $k+1$ in the denominator is uniformly $O(1/k^2)$, therefore the difficulty can only arise near the point $\phi=\omega_1$ when one takes the sum over $k$ of the following terms:
\begin{multline*}
    \left[
    \frac1{\sin\omega_1}I
    -\frac1{\sin\phi}
    \left(
      \begin{array}{cc}
        e^{2i(\phi-\omega_1)} & 0 \\
        0 & e^{-2i(\phi-\omega_1)} \\
      \end{array}
    \right)
    \right]
	\\
	\times
    \frac{|c|}{4k}
    \left(
      \begin{array}{cc}
        0 & e^{i(2(\phi-\omega_1)k-\delta_1)} \\
        e^{-i(2(\phi-\omega_1)k-\delta_1)} & 0 \\
      \end{array}
    \right).
\end{multline*}
Expression in the square brackets does not depend on $k$ and is $O(\phi-\omega_1)$ as $\phi\rightarrow\omega_1$, which cancels the zero of $\xi$ when we apply Lemma \ref{lem estimate of sum for Harris-Lutz}. This gives a uniform in $U_+$ estimate $T^+_n(\phi)=O(1/n)$ as $n\rightarrow\infty$. To obtain the equality \eqref{Harris-Lutz transformation T-pm} we use the estimate $e^Y=I+Y+O(\|Y\|^2)$ as $\|Y\|\rightarrow0$ together with \eqref{estimates of T-pm}: substitute $M_n$ in the form \eqref{M},\eqref{V-1} and \eqref{R-1} and $T^+_n$ in the form \eqref{T-pm} into the expression
$(I-T^+_{n+1}+O(1/n^2))M_n(I+T^+_n+O(1/n^2))$,
open the brackets, simplify the result and leave only the terms of the order $1/n$ (the smaller terms should be included into the remainder $R^+_n$). The remainder is uniformly summable and continuous in $U_+$, which follows immediately.
\end{proof}

Transform the system further using the similarity relation
\begin{equation*}
    \left(
       \begin{array}{cc}
         0 & e^{i\alpha} \\
         e^{-i\alpha} & 0 \\
       \end{array}
     \right)
     =
     \left(
       \begin{array}{cc}
         1 & i \\
         1 & -i \\
       \end{array}
     \right)
     \left(
       \begin{array}{cc}
         \cos\alpha & \sin\alpha \\
         \sin\alpha & -\cos\alpha \\
       \end{array}
     \right)
     \left(
       \begin{array}{cc}
         1 & i \\
         1 & -i \\
       \end{array}
     \right)^{-1}.
\end{equation*}
It is easy to check that by the transformation
\begin{equation*}
    \hat v^+_n:=
    \left(%
    \begin{array}{cc}
    1 & i \\
    1 & -i \\
    \end{array}%
    \right)^{-1}
    \left(%
    \begin{array}{cc}
    e^{i\delta_1/2} & 0 \\
    0 & e^{-i\delta_1/2} \\
    \end{array}%
    \right)
    \exp(-T^+_n(\phi))v_n
\end{equation*}
the system $v_{n+1}=M_n(\phi)v_n$ is reduced to the following one:
\begin{equation}\label{system for v-hat+}
    \hat v^+_{n+1}=
    \left[
    I+\frac{|c|}{4n\sin\omega_1}
    \left(%
    \begin{array}{cc}
    \cos(2(\phi-\omega_1)n) & \sin(2(\phi-\omega_1)n) \\
    \sin(2(\phi-\omega_1)n) & -\cos(2(\phi-\omega_1)n) \\
    \end{array}%
    \right)
    +\hat R_n^+(\phi)
    \right]
    \hat v^+_n,
\end{equation}
where $\{\hat R_n^+(\phi)\}_{n=1}^{\infty}\in l^1(U_+)$ and the function $\hat R^+_n(\cdot)$ is continuous in $U_+$ for every $n$. System \eqref{system for v-hat+} is equivalent for $\phi\in U_+$ to the eigenfunction equation \eqref{spectral equation} for the operator $\mathcal J$. Define the solution $\hat p^+(\phi)$ of \eqref{system for v-hat+} which corresponds to orthogonal polynomials:
\begin{multline}\label{p-hat+}
    \hat p^+_n(\phi):=
    \left(%
    \begin{array}{cc}
    1 & i \\
    1 & -i \\
    \end{array}%
    \right)^{-1}
    \left(%
    \begin{array}{cc}
    e^{i\delta_1/2} & 0 \\
    0 & e^{-i\delta_1/2} \\
    \end{array}%
    \right)
    \exp(-T^+_n(\phi))
    \\
    \times
    \left(%
    \begin{array}{cc}
    e^{-i\phi n} & e^{i\phi n} \\
    e^{-i\phi(n+1)} & e^{i\phi(n+1)} \\
    \end{array}%
    \right)^{-1}
    \left(
      \begin{array}{c}
        P_n(2\cos\phi) \\
        P_{n+1}(2\cos\phi) \\
      \end{array}
    \right).
\end{multline}
Now we are able to restate Proposition \ref{prop Janas-Simonov} in a form which is more convenient for our needs. The objects $\hat v^-_n,\hat R_n^-(\phi),\hat p^-_n(\phi)$ are defined in the same fashion for $\phi\in U_-$.

    \begin{lem}\label{lem Weyl-Titchmarsh formula}
    For every $\phi\in U_+$ the sequence $\{\hat p^+_n(\phi)\}_{n=1}^{\infty}$ given by \eqref{p-hat+} is a solution of the system 			\eqref{system for v-hat+}. For every $\phi\in U_+\backslash\{\omega_1\}$ it has a non-zero limit
    \begin{equation*}
    \lim_{\ninf}\hat p^+_n(\phi)=:\hat p^+_{\infty}(\phi).
    \end{equation*}
    The spectral density of $\mathcal J$ can be expressed in terms of this limit as
    \begin{equation}\label{rho' in U_+}
        \rho'(2\cos\phi)=\frac1{4\pi\sin\phi\l\|\hat p^+_{\infty}(\phi)\r\|^2},\ \phi\in U_+.
    \end{equation}
    Analogous statement holds true, if one replaces $\hat p^+$ by $\hat p^-$, $\hat R^+$ by $\hat R^-$, $U_+$ by $U_-$ and $\omega_1$ 	by $\pi-\omega_1$.
    \end{lem}

\begin{proof}
The assertion of Proposition \ref{prop Janas-Simonov} for $\phi\in U_+\backslash\{\omega_1\}$ can be rewritten as
\begin{equation*}
        \left(%
    \begin{array}{cc}
    e^{-i\phi n} & e^{i\phi n} \\
    e^{-i\phi(n+1)} & e^{i\phi(n+1)} \\
    \end{array}%
    \right)^{-1}
    \left(
      \begin{array}{c}
        P_n(2\cos\phi) \\
        P_{n+1}(2\cos\phi) \\
      \end{array}
    \right)
    \rightarrow
    \frac1{2\sin\phi}
    \left(
      \begin{array}{c}
        iF(e^{i\phi}) \\
        \overline{iF(e^{i\phi})} \\
      \end{array}
    \right)
\end{equation*}
as $n\rightarrow\infty$. Together with the fact that $T^+_n(\phi)=o(1)$ this yields:
\begin{equation*}
    \hat p^+_{\infty}(\phi)=
    \left(%
    \begin{array}{cc}
    1 & i \\
    1 & -i \\
    \end{array}%
    \right)^{-1}
    \left(%
    \begin{array}{cc}
    e^{i\delta_1/2} & 0 \\
    0 & e^{-i\delta_1/2} \\
    \end{array}%
    \right)
    \frac1{2\sin\phi}
    \left(
      \begin{array}{c}
        iF(e^{i\phi}) \\
        \overline{iF(e^{i\phi})} \\
      \end{array}
    \right)
\end{equation*}
An explicit calculation shows that
\begin{equation*}
    \left\|\hat p^+_{\infty}(\phi)\right\|=\frac{|F(e^{i\phi})|}{2\sin\phi}.
\end{equation*}
By Proposition \ref{prop Janas-Simonov} again,
\begin{equation*}
    \rho'(2\cos\phi)=\frac{\sin\phi}{\pi|F(e^{i\phi})|^2}=\frac1{4\pi\sin\phi\l\|\hat p^+_{\infty}(\phi)\r\|^2},
\end{equation*}
which completes the proof for $\phi\in U_+$. In the second case the proof is analogous.
\end{proof}

\section{Results for the model problem}\label{section Naboko-Simonov}

In this section we formulate results concerning the model system
\begin{equation}\label{model system}
    x_{n+1}=
    \left[
    I+\frac{\beta}n\left(%
    \begin{array}{cc}
    \cos(\varepsilon n) & \sin(\varepsilon n) \\
    \sin(\varepsilon n) & -\cos(\varepsilon n) \\
    \end{array}%
    \right)
    +R_n(\varepsilon)
    \right]
    x_n,\ n\in\mathbb N,\ \varepsilon\in U,
\end{equation}
which were obtained in \cite{Naboko-Simonov-2011}. Using these results we immediately get information about the behavior of the functions $\hat p^+_{\infty}(\phi)$ and $\hat p^-_{\infty}(\phi)$ near the points $\omega_1$ and $\pi-\omega_1$, respectively, and therefore about the behavior of the spectral density of $\mathcal J$ near the critical points, cf. \eqref{rho' in U_+}. Here $\beta$ is positive, $\varepsilon\in U$ is a small parameter, $U$ is an interval such that \begin{equation*}
    0\in U\subset (-2\pi;2\pi)
\end{equation*}
and the matrices $R_n(\varepsilon)$ are supposed to be uniformly summable in $n$ with respect to $\varepsilon\in U$ and continuous in $U$ for every $n$.

Let us write the system \eqref{model system} as
\begin{equation*}
    x_{n+1}=B_n(\varepsilon)x_n
\end{equation*}
with
\begin{equation}\label{B-n}
    B_n(\varepsilon):
    =
    I+\frac{\beta}n
    \left(%
    \begin{array}{cc}
    \cos(\varepsilon n) & \sin(\varepsilon n) \\
    \sin(\varepsilon n) & -\cos(\varepsilon n) \\
    \end{array}%
    \right)
    +R_n(\varepsilon).
\end{equation}
We parametrize different solutions by their initial conditions $f\in\mathbb C^2$ (while the system itself depends on the small parameter $\varepsilon\in U$):
\begin{equation}\label{u(f)}
    \begin{array}{rl}
    x_1(\varepsilon,f)&:=f,
    \\
    x_{n+1}(\varepsilon,f)
    &:=
    B_n(\varepsilon)x_n(\varepsilon,f),\,n\ge1.
    \end{array}
\end{equation}

    \begin{prop}[Naboko-Simonov]\label{prop Naboko-Simonov}
    Assume that functions $R_n(\cdot)$ are continuous in $U$ for every $n\in\mathbb N$, the matrices $B_n(\varepsilon)$ are invertible for every $n\in\mathbb N$, $\varepsilon\in U$ and the sequence $\{R_n(\varepsilon)\}_{n=1}^{\infty}\in l^1(U)$.
    Then for every $f\in\mathbb C^2$ and every $\varepsilon\in U\backslash\{0\}$ the limit
    \begin{equation*}
        \lim_{n\rightarrow\infty}x_n(\varepsilon,f)
    \end{equation*}
    exist. For $\varepsilon=0$ the limit
    \begin{equation*}
        \lim\limits_{n\rightarrow\infty}\frac{x_n(0,f)}{n^{\beta}}
    \end{equation*}
    exists for every $f$ and the linear map
    \begin{equation*}
        f\mapsto \lim\limits_{n\rightarrow\infty}\frac{x_n(0,f)}{n^{\beta}}
    \end{equation*}
    has rank one. If, moreover, $f$ is such that $\lim\limits_{n\rightarrow\infty}\frac{x_n(0,f)}{n^{\beta}}\neq0$, then there exist two one-side limits
    \begin{equation*}
        \lim\limits_{\varepsilon\rightarrow\pm0}|\varepsilon|^{\beta}\lim\limits_{n\rightarrow\infty}x_n(\varepsilon,f)\neq0.
    \end{equation*}
    \end{prop}

This result can be reformulated in terms of infinite matrix products, which will be more useful for us here. Let $R$ stand for the whole sequence $\{R_n(\varepsilon)\}_{n=1}^{\infty}$.

    \begin{prop}\label{prop matrix products}
    In assumptions of Proposition \ref{prop Naboko-Simonov}, the following holds. For every $\varepsilon\in U\backslash\{0\}$ there exists
    \begin{equation*}
        \Phi(\beta,\varepsilon,R):=\prod_{n=1}^{\infty}B_n(\varepsilon).
    \end{equation*}
    For $\varepsilon=0$ there exists the limit
    \begin{equation*}
        \Phi_0(\beta,R):=\lim_{N\rightarrow\infty}\frac1{N^{\beta}}\prod_{n=1}^NB_n(0),
    \end{equation*}
    which is a matrix of rank one. And finally, there exist two one-side limits
    \begin{equation*}
        \Phi_{\pm}(\beta,R):=
        \lim_{\varepsilon\rightarrow\pm0}|\varepsilon|^{\beta}\Phi(\beta,\varepsilon,R)
    \end{equation*}
    such that
    \begin{equation*}
        \text{Ker }\Phi_0(\beta,R)=\text{Ker }\Phi_-(\beta,R)=\text{Ker }\Phi_+(\beta,R).
    \end{equation*}
    \end{prop}

\section{Zeroes of the spectral density}\label{section final result}

In this section we put together all the ingredients: the Weyl-Titchmarsh type formula from \cite{Janas-Simonov-2010}, the analysis of \cite{Naboko-Simonov-2011} and the transformations of Section \ref{section reduction}, to obtain the main result of the present paper.

    \begin{thm}\label{thm main result}
    Let $\mathcal J$ be the discrete Schr\"odinger operator with the potential
    \begin{equation*}
    \frac{c\sin(2\omega n+\delta)}n+q_n,
    \end{equation*}
    where $c,\omega,\delta$ are real constants, $\{q_n\}_{n=1}^{\infty}$ is a real-valued sequence such that
    \begin{equation*}
    c\neq0,\omega\notin\frac{\pi\mathbb Z}2\text{ and }\{q_n\}_{n=1}^{\infty}\in l^1.
    \end{equation*}
    Let $\nu_{cr}\in\{-2\cos\omega,2\cos\omega\}$. If $\nu_{cr}$ is neither an eigenvalue nor a half-bound state of $\mathcal J$, then there exist two one-side    limits
    \begin{equation*}
        \lim_{\lambda\rightarrow\nu_{cr}\pm0}\frac{\rho'(\lambda)}
        {|\lambda-\nu_{cr}|^{\frac{|c|}{2|\sin\omega|}}},
    \end{equation*}
    where $\rho'$ is the spectral density of $\mathcal J$.
    \end{thm}

\begin{proof}
Consider the neighbourhood of the critical point $2\cos\omega_1$ and $\phi\in U_+$. Take
\begin{equation*}
    \varepsilon:=2(\phi-\omega_1),
\end{equation*}
see \eqref{system for v-hat+}. Lemma \ref{lem Weyl-Titchmarsh formula} yields in the notation of Proposition \ref{prop matrix products}:
\begin{multline*}
    \rho'(\lambda)=\rho'(2\cos\phi)=\rho'(2\cos(\omega_1+\varepsilon/2))
    \\
    =\frac1{\pi\sin(\omega_1+\varepsilon/2)}
    \frac1{
    \left\|
    \Phi\left(\frac{|c|}{4\sin\omega_1},\varepsilon,\hat R^+\right)
    \hat p^+_1(\omega_1+\varepsilon/2)
    \right\|^2}.
\end{multline*}
Now the asymptotics of $\rho'(2\cos(\omega_1+\varepsilon/2))$ as $\varepsilon\rightarrow\pm0$ follow from Proposition \ref{prop matrix products} due to the continuity of $\hat p^+_1(\phi)$, if only $$\hat p^+_1(\omega_1)\notin\text{Ker}\,\Phi_0\left(\frac{|c|}{4\sin\omega_1},\hat R^+\right).$$
The latter by the definition of $\Phi_0$ means  that $$\lim\limits_{\ninf}\frac{\hat p^+_n(\omega_1)}{n^{\frac{|c|}{4\sin\omega_1}}}\neq0.$$
Due to the relation \eqref{p-hat+} and Theorem \ref{thm asymptotics of GEV} this in turn means that orthogonal polynomials at the point $2\cos\omega_1$ are not $O\left(n^{-\frac{|c|}{4\sin\omega_1}}\right)$ as $\ninf$, see the formulas \eqref{asymptotics at omega} and \eqref{asymptotics at -omega}.
In the opposite case, the point $2\cos\omega_1$ is either an eigenvalue of $\mathcal J$ (if $\frac{|c|}{4\sin\omega_1}>\frac12$) or a half-bound state (if $\frac{|c|}{4\sin\omega_1}\le\frac12$). Since $\sin\omega_1=|\sin\omega|$, this proves the result for the critical point $2\cos\omega_1$. The proof for the second critical point $-2\cos\omega_1$ can be obtained by changing $\omega_1$ to $\pi-\omega_1$ and $+$ to $-$ in the notation.
\end{proof}

\subsection*{Acknowledgements}
The author expresses his deep gratitude to Prof. S.N. Naboko for his constant attention to this work and for many fruitful discussions of the subject, to Prof. Jan Janas for invaluable help in the work and to Prof. Harald Woracek for many important remarks and suggestions. The work was supported by the Chebyshev Laboratory (Department of Mathematics and Mechanics, Saint-Petersburg State University) under the grant 11.G34.31.0026 of the Government of the Russian Federation, by grants RFBR-09-01-00515-a and 11-01-90402-Ukr\_f\_a and by the Erasmus Mundus Action 2 Programme of the European Union.

\end{document}